\documentclass[12pt,reqno]{amsart}

\font\bbb=msbm10 scaled \magstep1
\font\fff= eufm10 scaled \magstep 1

\newtheorem{theorem}{Theorem}[section]   

\newtheorem*{GRCP}{Generalized Road Coloring Problem}   
\newtheorem{corollary}[theorem]{Corollary}     
\newtheorem{lemma}[theorem]{Lemma}         
\newtheorem{proposition}[theorem]{Proposition}  
\newenvironment{remark}{{\bf Remark.\ }\rm}{\bigskip}  
\newenvironment{example}{{\bf Example.\ }\rm}{\bigskip}
\newenvironment{definition}{{\bf Definition.\ }\rm}{\bigskip}
\newenvironment{acknowledgment}{\bigskip{\bf Acknowledgment.\ }\rm}{\bigskip}    

\def\vep{\varepsilon}

\def\RR{\hbox{\bbb R}}

\def\avg#1.{\langle #1 \rangle}

\def\comment#1{}

\def\Mat{{\rm Mat}}

\def\tr{{\rm tr}\,}
\def\avg#1.{\langle #1 \rangle}
\newcommand{\sym}[2]{\ensuremath{{#1}{#2}{#1}^*}}

\newcommand{\range}{\operatorname{range}}
\newcommand{\card}{\operatorname{card}}
\newcommand{\dist}{\operatorname{dist}}
\newcommand{\mat}{\operatorname{Mat}}

\def\K{{\mathcal K}}

\def\A{{\mathcal A }}

\def\V{{\mathcal V}}
\def\E{{\mathcal E }}

\def\?#1?{\hbox{\fff #1}}
\def\tsp{\thinspace}
\def\vstrut{{\phantom{\biggm|}}}

\def\dequiv{\underset{\delta}{\equiv}}

\def\Sequiv{\underset{S}{\equiv}}

\def\dg#1.{{#1}^\dagger}                                                   

\newcommand{\spow}[2]{\ensuremath{{#1}^{\vee #2}}}  
\def\hfb{\hfill\break}

\title[Periodic Digraphs, Generalized Road Coloring Problem]
{The Generalized Road Coloring Problem and Periodic Digraphs}
\author{G. Budzban and Ph. Feinsilver}
\address{Department of Mathematics\\
Southern Illinois University\\
Carbondale, IL\\
62901 USA}
\date{}

\begin{document}
\begin{abstract}
A proof of the Generalized Road Coloring Problem, independent of the recent work by Beal and Perrin,
is presented, using both semigroup methods and Trakhtman's algorithm.
Algebraic properties of periodic, strongly connected digraphs are studied in the semigroup context.
An algebraic condition which characterizes periodic, strongly connected digraphs is determined in the 
context of periodic Markov chains.
\end{abstract}
\maketitle

\thispagestyle{empty}

\section{Introduction}  
There is a rich history of research in automata theory concerning synchronizability. Two problems, in particular, have contributed significantly to this history: \u{C}ern\'{y}'s Conjecture and the Road Coloring Problem.  Professor Trakhtman's recent solution of the Road Coloring problem is the conclusion to more than thirty years of investigation, while \u{C}ern\'{y}'s Conjecture remains a source of continued intense interest.\bigskip

Both of these problems were originally stated for the case where the underlying state transition digraph is strongly connected and aperiodic, since aperiodicity is clearly a necessary condition for the existence of a synchronizing instruction.  Yet, each of these problems has an easily stated generalization to the case where the underlying digraph is periodic of degree greater than 1.  A private communication with Professor Trakhtman pointed us to a recent Arxiv
posting by Professors Beal and Perrin where they prove this ``Generalized Road Coloring Problem" (GRCP)
\cite{BP}.\bigskip

We start with an independent proof of the GRCP using a semigroup approach and Trakhtman's algorithm.
Properties of periodic digraphs and the structure of their accompanying ``coloring semigroups''
are analyzed using methods initiated in \cite{Bu} 
and the algebraic techniques introduced in \cite{BPF06}, \cite{BMu},  and \cite{Fr}. 
Properties of periodic digraphs are then studied in the context of periodic Markov chains.
We prove a theorem for determining if a graph is periodic, no easy task if one is merely provided an adjacency matrix,
that leads to a constructive way for finding the periodic classes as well.

\section{The structure of the kernel of a coloring semigroup}  

To begin the work of generalizing Trakhtman's result to the periodic case, we recall some basic notions from transformation semigroup theory.\bigskip

Start with $\A$, the adjacency matrix of a $d$-out strongly connected digraph, $G$.
A coloring of $G$ is a decomposition of $\A$ into $m$ binary (that 
is, $0\,\text{-}\,1$) stochastic matrices $R_1, R_2,\ldots, R_d$ such that $\A = R_1 + \cdots + R_d$. 
Intuitively, one can think of a matrix $R_{i}$ as assigning the $i$th of 
$d$ colors to the edge $(j,k)$ if $(R_i)_{jk} = 1$. 
Each $n\times n$ binary stochastic matrix $R_i$ can also be thought of as a function,
$\hat R_i$, say, on the vertices $\{1, 2,\ldots, n\}$, where $(R_i)_{jk} = 1$ iff $j\hat R_i = k$. \bigskip

Let $S = \langle R_1, R_2, \ldots, R_d \rangle$ be the semigroup generated by $R_1, R_2, \ldots, R_d$ under matrix multiplication (or composition). 
Clearly, $S$ is a finite semigroup which we will refer to as a coloring semigroup. The following classic theorem on the structure of the minimal ideal of a finite transformation semigroup can be found in many references, for example \cite{CP}.
Nothing which follows requires the assumption of aperiodicity. \bigskip 

\begin{remark} For $L\subset S$, $E(L)$ denotes the subset of idempotents in $L$.\end{remark}

\begin{theorem}
Let $S$ be a finite semigroup. 
Then $S$ contains a minimal ideal, $\K$, called the kernel which is a disjoint union of isomorphic groups. 
The kernel, $\K$, is isomorphic to $X \times G \times Y$,
 where, fixing $e \in  E(S)$, $e\K e$ is a group, $X = E(\K e)$, $G = e\K e$, and $Y = E(e\K)$. 
If $(x_1, g_1, y_1)$, $(x_2, g_2, y_2) \in X \times G \times Y$ then
\[
  (x_1, g_1, y_1)(x_2, g_2, y_2) 
	= (x_1, g_1(y_1x_2)g_2, y_2)\,.
\]
The product structure $X \times G \times Y$ is called a Rees product and any semigroup that has a Rees product is called completely simple. 
Thus the kernel of a finite semigroup is always completely simple.
\end{theorem}

Suppose $S = \langle R_1, R_2,\ldots, R_d \rangle$ is a coloring semigroup for a given $d$-out digraph $G$. 
As stated above, $S$ will be a finite semigroup with kernel $\K = X \times G \times Y$. 
It can be shown \cite{WEC} that each of the elements of $\K$ have the same rank. 
Thus we can refer to the rank of the kernel, rank$(\K)$. 
Let $M$, $N$ be elements of $\K$. 
Then with respect to the Rees product structure $M = (M_1,M_2,M_3)$ and $N = (N_1,N_2,N_3)$. 
We will consider the structure of $\K M$, $N\K$, and $N\K M$:\bigskip
\begin{enumerate}
\item[1)] $\K M = X \times G \times \{M_3\}$ is a minimal left ideal in $\K$ whose elements all have the same range, or nonzero columns as $M$.\bigskip

\item[2)] $N\K = \{N_1\} \times  G \times Y$ is a minimal right ideal in $\K$ whose elements all have the same partition of the vertices as $N$. 
If $P$ is an equivalence class in the partition, then there exists a nonzero column $j$ such that
\[
  P = \{i : N_{ij} = 1\}\,.
\]

\item[3)] $N\K M$ is the intersection of $N\K$ and $\K M$ and is a maximal group in $\K$. 
In this instance, it is best thought of as a set of one-to-one functions specified by the partition of $N$ and the range of $M$. 
The idempotent of $N\K M$ is the function which is the identity when restricted to the range of $M$.\bigskip
\end{enumerate}

Suppose $w = (w_1, w_2, \ldots, w_n)$ is such that $w\A = dw$. 
We may and do scale $w$ so that each $w_i$ is a positive integer and that they are relatively prime.
If $U$ is a subset of the vertices $V$, we define $W(U) = \sum\limits_{j\in U} w_j$ to be the 
\textsl{Friedman weight} of $U$. 
It can be shown that if $N\K$ is an arbitrary minimal right ideal in $\K$ with partition $\Pi = \{P_1,P_2,\ldots,P_r\}$, then for all $i$, $1 \le i \le r$, $w(P_i) = \dfrac{W(V)}{r}$. 
In the terminology of Friedman, each partition of the kernel consists of maximally synchronizing sets (see the proof of Theorem 7 in \cite{BMu}).\bigskip

Suppose a subset of vertices $B = \range(K)$ for some $K \in \K$. An important property of $B$ that will play a role in this paper is that $B$ is a \textit{cross-section} of every partition $\Pi$ from $\K$. That is, for each element $P$ of $\Pi$, card$(P \cap  B) = 1$ (see \cite{CP}).

\section{Periodic Graphs, Quotient Graphs, Cross Sections, and the structure 
	of the Kernel}  

\begin{subsection}{Periodic graphs} 

The most useful definition of a periodic graph is the following.\bigskip

\begin{definition}
A digraph $G = (V, E)$ is periodic of period $t \ge 2$ iff there exists a partition of the vertices 
$\{P_1, P_2, \ldots, P_t\}$ such that if $(i,j)\in  E$ and $i \in P_k$ then $j \in P_{k+1}$ ($i \in P_t$ implies $j \in P_1)$, and $t$ is the largest such integer with this property. 
In this case, $\Pi =\{P_1, P_2, \ldots, P_t\}$ is the ``periodic partition".
If no such partition exists, then $G$ is aperiodic. 

\end{definition}
\end{subsection}

\subsection{Quotient graphs, cross sections, and the structure of the kernel}

In \cite{BPF07}, quotient graphs generated by the unique partition induced by a right group kernel in a coloring semigroup were analyzed. 
In that paper, it was shown that if the kernel of any coloring semigroup is a right group, then the graph has some coloring semigroup that has a synchronizing instruction. 
The algorithm designed by Kari in \cite{Ka} was essential to this argument. 
Perhaps the most important concept introduced by Kari was the notion of stability. 
Kari showed that stability induced a congruence on the vertices 
of the graph, and it was this congruence that produced the  quotient graph instrumental in his result.\bigskip

We will translate Kari stability into the language of transformation semigroups.\bigskip

\begin{definition}
Let $S = \langle R_1, R_2,\ldots, R_d \rangle$ be a coloring semigroup for the digraph $G = (V, E)$. 
Then for vertices $x$, $y$, we say $x\Sequiv y$ under the stability relation if and only if for every $W_1 \in S$, there exists a $W_2 \in S$ such that $xW_1 W_2 = yW_1 W_2$.\bigskip

If $\K = X \times G \times Y$ is the kernel of a coloring  semigroup $S$, then 
$\displaystyle\K = \bigcup_{N\in X} N \K$, 
where each $N\K$ is a right group with a fixed partition $\Pi_N$. 
Let $\E_N$ be the equivalence relation induced by $\Pi_N$. Finally let $\E = \cap  \E_N$.
\end{definition}

\begin{theorem} 
For vertices $x$, $y$ in $V$, $x\E y$ if and only if $x\Sequiv y$.
\end{theorem}

\begin{proof} 
Suppose $x\E y$. 
Then for any partition $\Pi$ from the kernel, there exists an element $P$ of $\Pi$, such that $\{x,y\} \subset P$. 
Thus for any $K \in \K$, $xK = yK$. 
Let $W \in S$ be any word, then $xWK = yWK$ since $WK \in \K$. Therefore, $x\Sequiv  y$.\bigskip

Suppose $x$ and $y$ are not equivalent under $\E$. 
Then there is some idempotent $N$ in $\K$ with partition $\Pi_N$ containing elements $P_1 \ne P_2$ such that $x \in P_1$ and $y \in P_2$, which is to say that $xN \ne yN$. 
But for all words $W \in S$, $NW \in N\K$. 
But this implies that $xNW \ne yNW$. 
Therefore $x$ and $y$ are not equivalent under $\Sequiv$.
\end{proof}

Since Kari showed \cite{Ka} that stability is a congruence, it follows from the previous theorem that a well-defined quotient graph (with less vertices than the original graph) exists precisely when $\E = \cap \E_N$ is not the discrete equivalence.\bigskip

\begin{corollary}  
Let $G$ be a periodic graph with periodic partition $\Pi=\{P_1,\ldots,P_t\}$.  
Let $[v]$ be an equivalence class of vertices induced by the stability relation.  
Then $[v] \subset P_k$ for some $k$, $1\le k\le t$.
\end{corollary}

\begin{proof}  
If $x \Sequiv  y$ then $xK = yK = z$ for some $K \in \K$.
Thus a path of length $|K|$ connects both $x$ and $y$ to $z$, and therefore $x$ and $y$ must have started in the same periodic partition class.
\end{proof}

In \cite{Tr}, one of the key notions utilized by Trakhtman is that of an $F\,\text{-}\,$clique.\bigskip

\begin{definition}
Let $B = \range(W)$ for some $W \in S$. 
The $B$ is called an $F\,\text{-}\,$clique 
if and only if for any pair $x \ne y$ in $B$, $xU \ne yU$ for all words $U \in S$. 
\end{definition}

Once again, the structure of the kernel provides insight.

\begin{theorem} 
A set of vertices $B=\range(W)$, for some $W\in S$,  is an $F\,\text{-}\,$clique if and only if $W \in \K$. 
\end{theorem}

\begin{proof} 
Suppose $B = \range(K)$ where $K \in \K$. 
Then, as indicated above, $B$ is a cross-section of each partition. 
Suppose for some pair $x \ne y$ in $B$, and some $U \in S$, 
$xU = yU$. 
Then $xUK = yUK$, implying that $\{x,y\} \subset P$, for some $P\in \Pi_{UK}$. 
But this contradicts that $B$ is a cross-section of $\Pi_{UK}$.\bigskip

Suppose $B = \range(W)$ for some $W \notin \K$. 
Since $W \notin \K$, then $\card(\range W)$ is not minimal. 
Then $\card (BK) < \card(\range W)$ since $WK \in \K$. 
But this implies that for some pair $x$, $y$ in $B$, $xK=yK$. Thus $B$ is not an $F\,\text{-}\,$clique.
\end{proof}

The key insight of this section is that critical pieces of both Kari's and Trakhtman's ideas can be restated in transformation semigroup language, and in this setting it becomes clear that nothing about them requires the assumption of aperiodicity. 

\section{The Generalized Road Coloring Problem}  

To complete the proof using Trakhtman's algorithm, we will use the following result which will allow us to complete the necessary inductive argument.\bigskip

\begin{theorem}  
Let $G$ be a graph with per$(G) = t$.  Let $Q(G)$ be a quotient graph of $G$ induced by the stability classes of a coloring semigroup.  
Then per$(Q(G)) = t$.
\end{theorem}

\begin{proof}  
Let per$(G) = t$ and let $\Pi = \{P_1,\ldots,P_t\}$ be the periodic partition of the vertices such that for any edge $(v,w)$ if $v \in P_k$, then $w \in P_{k+1}$  (here $P_{t+1} = P_1$).\bigskip

Let $C(G)$ be the cycles of $G$ and let $c$ be an element of $C(G)$  starting and ending at $v$.  
Then $[c]$ is an element of $C(Q(G))$ that starts and ends at $[v]$, where $[v]$ is the equivalence class of $v$.  
Thus $|C(G)| \subset |C(Q(G))|$, where $|C(G)| = \{|c| : c \in C(G)
\}$. 
Therefore the set of common divisors of $C(Q(G))$ is a subset of the set of common divisors of $C(G)$.  
But the period is the greatest of these common divisors, so that per$(Q(G)) \le \hbox{per}(G)$.\bigskip

Now let $c$ be a cycle in $Q(G)$ starting and ending at $[v]$.  
But $[v] \subset P_k$ for some $k$.  Thus $|c| = nt$ for some $n$. 
But then $t = \hbox{per}(G)$ is a divisor of all cycle lengths for $Q(G)$, which implies that per$(G) \le \hbox{per}(Q(G))$.  
Thus $t = \hbox{per}(G) = \hbox{per}(Q(G))$.
\end{proof}
\bigskip

The analogous result concerning the quotient graph being strongly connected is clear.\bigskip

In \cite{BMu}, the following generalization of the Road Coloring Problem was stated.\bigskip

\begin{GRCP}
Let $G = (V,E)$ be a strongly connected $d$-out digraph. 
Then $G$ is periodic of period $t \ge 1$ with periodic partition $\Pi = \{P_1, \ldots, P_t\}$ iff $G$ has a minimal (with respect to the rank of its kernel) coloring semigroup $S = \langle R_1\ldots, R_d \rangle$ whose kernel $\K$ is a right group with partition $\Pi$ and whose maximal groups are all cyclic of order $t$.
\end{GRCP}

\begin{proof}  
Observe that the constant functions form a right group with partition  $\Pi = \{V\}$ and whose maximal groups are all order one, thus trivially cyclic.  
In [2], it was shown that the existence of a coloring semigroup whose kernel $\K$ is a right group where rank$(\K)$ is minimal implies that the graph $G$ is periodic of order rank$(\K)$ and that the structure groups of $\K$ are cyclic of order rank$(\K)$.\bigskip

Notice, if $S$ is a coloring semigroup whose kernel $\K$ is a right group where rank$(\K) = \hbox{per}(G)$, then any element of $\K$ maps the set of vertices onto a $t$-subset of vertices, one vertex in each of the periodic partition classes, and this is the best one can do.  Thus we will refer to existence of this minimal right group as ``$t$-synchronizability''. \bigskip

For the converse, we will proceed by induction.  
Now the smallest $t$-periodic graph has $t$ vertices with each vertex in its own periodic partition class.  
Clearly, any coloring of this graph produces a coloring semigroup whose kernel is a cyclic group (thus a right group) of order $t$, and thus satisfies the conclusion of the theorem.  
Assume that all $t$-periodic graphs with $n-1$ vertices or less have been shown to be $t$-synchronizable, and let $G$ be $t$-periodic with $n$ vertices. \bigskip

Let $S_{\min} = \{S_1,\ldots,S_r\}$ be the collection of coloring semigroups of $G$ whose kernels $\{\K_1,\ldots,\K_r\}$ have minimal rank.  
If for some $j$, $\K_j$ is a right group, then rank$(\K_j) = t$, and $G$ is $t$-synchronizable.  
Thus we may assume that none of the kernels of minimal rank are right groups.  
If any of the kernels have non-trivial stability classes, then their quotients are strongly connected and $t$-periodic, and by the induction hypothesis are $t$-synchronizable.  
Thus suppose all of the minimal rank kernels have trivial stability classes.  
In this case, Trakhtman's remarkable algorithm, which at no point uses the assumption of aperiodicity, produces a coloring with  nontrivial stability class, and thus a contradiction, and we are done. 
\end{proof}

\section{Properties of Periodic Markov Chains} 

Let $G$ be a $d$-out digraph with  adjacency matrix $\A$.
Dividing through by the outdegree produces a stochastic matrix, $A$.  
The question is:\medskip

Is there a condition on $A$ that determines if the graph is periodic, and can one use this condition to determine the periodic partition of $G$? \medskip

We will assume that $A$ is irreducible or, equivalently, 
that $G$ is strongly connected.\bigskip

We will consider stochastic matrices in general.
So for the remainder of this work, we assume that $A$ is a stochastic matrix that is the transition
matrix for an irreducible Markov chain.\bigskip

The number of vertices is $\card V=n$. Denote by $\V$ the corresponding $n$-dimensional vector space
$\RR^n$.\bigskip

\begin{remark}We use the summation convention that Greek indices are summed over
regardless of position.
\end{remark}


\subsection{Level 2 action. Aperiodic case.}
\bigskip

\begin{definition} Let ${\rm Sym}(\V)$ be the cone of $n\times n$ real symmetric matrices with nonnegative entries.
Denote by ${\rm Sym}_0(\V)$ the cone of $n\times n$ symmetric matrices having nonnegative entries 
with all diagonal entries equal to zero. 
\end{definition}

We consider the action $X\mapsto AXA^*$ on ${\rm Sym}(\V)$. This is the {\sl level 2 action\/} of the transition
matrix $A$.\medskip

\begin{subsubsection}{Aperiodic case}

Begin with the lemma:

\begin{lemma}\label{lem:irrep}
Let $A$ be an irreducible, aperiodic stochastic matrix. Then, on ${\rm Sym}_0(\V)$, the equation
$AXA^*=X$ has the unique solution $X=0$.
\end{lemma}
\begin{proof}
Iterating the map $X\mapsto AXA^*$ yields $X=A^mX(A^*)^m$ after $m$ steps. Since $A$ is irreducible and
aperiodic, $\displaystyle \lim_{m\to\infty} A^m=\Omega$ exists where $\Omega$ is a rank-one stochastic
matrix with positive entries. Thus, taking $m\to\infty$, we have $X=\Omega X\Omega^*$. Taking traces yields:
$$ 0=\tr X=\tr X\Omega^*\Omega=\sum_{i,j}X_{ij}(\Omega^*\Omega)_{ji}$$
Now, $(\Omega^*\Omega)_{ij}>0$, $\forall i,j$, so $X_{ij}$ vanishes for all $i,j$.
\end{proof}
\end{subsubsection}

\begin{subsection}{Periodic case}
Recall the periodic partition $\{P_1, P_2, \ldots, P_t\}$. For vertices $i,j\in V$, define
$$\dist(i,j)=\min\{\,|s_1-s_2|,|s_1-s_2-t|\,\} \qquad\text{for\ }i\in P_{s_1}, j\in P_{s_2}$$
i.e., it is the smallest difference between the indices of the partition classes they are in, taken modulo $t$.
Define the corresponding indicator matrices $X^{(\delta)}$ by 
$$X^{(\delta)}_{ij}=\begin{cases}1, & \text{if $\dist(i,j)=\delta$} \cr 0, & \text{otherwise} \cr\end{cases}
$$
These are $0\,\text{-}\,1$ matrices with 
\begin{equation}\label{eq:ortho}
\tr X^{(\delta_1)}X^{(\delta_2)}=0
\end{equation}
 if $\delta_1\ne\delta_2$.
As they correspond to a partition of all pairs of vertices, we have $\displaystyle \sum_\delta X^{(\delta)}=J$,
the all-ones matrix. We denote the corresponding equivalence relations
\begin{align*}
&& i\dequiv j &&\text{ if } \dist(i,j)=\delta
\end{align*}
So $i\underset{0}{\equiv} j$ if and only if $i$ and $j$ are in the same partition class of the periodic
partition. The periodicity of $A$ is expressed by the property that 
$$\text{if } A_{ii'}A_{jj'}>0, \text{ then } i\dequiv j \text{ if and only if } i' \dequiv j'$$
since each vertex maps to the next partition class. Effectively, $A$ acts isometrically on the partition classes.\bigskip

\begin{remark}For clarity, we will not use implied summations over $\delta$, repeated or not.
\end{remark}

Observe that $X^{(0)}\in{\rm Sym}(\V)$, with $1$'s on the diagonal, while \hfb
$X^{(\delta)}\in{\rm Sym}_0(\V)$, for $\delta>0$.
\begin{proposition}\label{prop:sols}
For $\delta\ge0$, $X^{(\delta)}$ satisfies $AX^{(\delta)}A^*=X^{(\delta)}$.
\end{proposition}
\begin{proof} The entries of $AXA^*$ are
$$(AXA^*)_{ij}=A_{i \lambda}X_{\lambda \mu}A_{j\mu}$$
Fix $\delta\ge0$. Consider $i,j$ such that $i\dequiv j$.
For every pair $(l,m)$ such that $i\to l$ and $j\to m$, $l\dequiv m$ as well. Thus, for all such pairs,
$X^{(\delta)}_{lm}=1$ and, with $u\in\V$ denoting the vector of all $1$'s,
$$(AX^{(\delta)}A^*)_{ij}=A_{i \lambda}u_\lambda A_{j\mu}u_\mu=1=X^{(\delta)}_{ij}$$
For $i\not\underset{\delta}{\equiv}j$, $X^{(\delta)}_{lm}=0$ for all choices of $l$ and $m$ such that
$A_{il}A_{jm}>0$, so that $(AX^{(\delta)}A^*)_{ij}=0$. Thus, 
$AX^{(\delta)}A^*=X^{(\delta)}$ as required.
\end{proof}

The main property we need is this.

\begin{lemma}\label{lem:SW}
Let $X\in{\rm Sym}(\V)$ satisfy $\sym{A}{X}=X$. 
If $X_{i_0j_0}=c$ for any one entry $i_0\dequiv j_0$, then $X_{ij}=c$ for all $i\dequiv j$.
\end{lemma}
\begin{proof}
From irreducibility, with period $t$, we have
$$\lim_{m\to\infty} (A^t)^m=\Omega_t$$
exists and if $i,j$ are in the same partition class, $(\Omega_t)_{ij}>0$, \cite[pp.\tsp 177-178]{D}. 
First, iterating with $A$, we have $\sym{A^t}{X}=X$, then iterating with $A^t$, we have
$A^{mt}X(A^{mt})^*=X$. Taking the limit as $m\to \infty$ yields
\begin{equation}\label{eq:limit}
\sym{\Omega_t}{X}=X \qquad \text{ or } \quad X_{ij}=(\Omega_t)_{i\lambda}(\Omega_t)_{j\mu}X_{\lambda\mu}
\end{equation}
Find the minimum nonzero entry  of $X_{ij}$ where $i\dequiv j$. Call it $c_1$, with $X_{i_1j_1}=c_1$. We have
$i_1\in P_a$, $j_1\in P_b$, where with no loss of generality $b=a+\delta \mod{t}$.
Then $Y=X-c_1X^{(\delta)}$ is a nonnegative solution to $\sym{A}{Y}=Y$ with 
$Y_{i_1j_1}=0$.  Applying eq. \eqref{eq:limit} for $Y_{i_1j_1}$ we have
$$0=(\Omega_t)_{i_1\lambda}(\Omega_t)_{j_1\mu}Y_{\lambda\mu}$$Since $(i_1,j_1)\in P_a\times P_b$, every pair $(\lambda,\mu)\in P_a\times P_b$.  Now, the $\Omega_t$ factors are all positive, and each term of the sum is nonnegative. Hence, the terms vanish identically. I.e., $Y_{ij}=0$ for all
$(i, j)\in P_a\times P_b$. \medskip

It is possible that $P_b$ is the only partition class $\delta$ units away from $P_a$.
Typically, however, the class $\delta$ units going the other way, $P_{a-\delta}$, indices mod $t$, is not
$P_b$. Call this class that precedes $P_a$ by distance $\delta$, $P'$. 
Moving cyclically along the partition classes, choose the power $s$ so that $A^s$ maps $P_a$ to $P'$, automatically
 mapping $P_b$ to $P_a$. We now have
$$ \Omega_t A^sY(A^*)^s\Omega_t^*=Y \quad \text{ and }\quad
0=(\Omega_t)_{i_1\lambda}(A^s)_{\lambda \sigma}Y_{\sigma\vep}(A^s)_{\mu\vep}(\Omega_t)_{j_1\mu}
$$
with $i_1\to \lambda\to \sigma$ and $j_1\to \mu\to \vep$. Notice the pairs $(i_1,j_1)\in P_a\times P_b$,
$(\lambda,\mu)$ run through all elements of $P_a\times P_b$, and
$(\sigma,\vep)\in P'\times P_a$. As before, all the $\Omega_t$
factors are positive. Since $A$ is irreducible, $A$ maps each partition class {\sl onto\/} the succeeding one.
So $A^s$ maps $P_a$ onto $P'$ and $P_b$ onto $P_a$. Thus all possible choices $(\sigma,\vep)\in P'\times P_a$
occur with positive $A$ coefficients for some choices $(\lambda,\mu)$. 
Hence the $Y$ terms vanish identically and $Y_{ij}=0$ on
$P'\times P_a$. By symmetry of $Y$, $Y_{ij}$ vanishes as well on $P_b\times P_a$ and $P_a\times P'$.
In other words, $Y_{ij}=0$ if $i\dequiv j$.\medskip

Going back to $X$, in fact, $X_{i_0j_0}=c=X_{i_1j_1}=c_1=X_{ij}$ for all $i\dequiv j$ as required.
\end{proof}

We noticed above that the matrices $X^{(\delta)}$ are orthogonal for different $\delta$, in the sense
of equation \eqref{eq:ortho}. 

\begin{lemma}\label{lem:decomp}
Let $X\in{\rm Sym}(\V)$ satisfy $AXA^*=X$. Then
$X-cX^{(\delta)}$, where 
$$c=\tr XX^{(\delta)}/\tr (X^{(\delta)})^2$$
is in ${\rm Sym}(\V)$.
\end{lemma}
\begin{proof} Looking at $ij$ entries, we wish to show that
$$X_{ij}( X_{\lambda \mu}^{(\delta)}X_{\lambda \mu}^{(\delta)})
\ge (X_{\lambda \mu}^{(\delta)}X_{\lambda \mu})X^{(\delta)}_{ij}$$
where there is no implied summation over $\delta$. If $i\not\underset{\delta}{\equiv}j$, 
then the right-hand side is zero, so the inequality is satisfied. For $i\underset{\delta}{\equiv}j$,
the right-hand side equals
$$\sum_{l\dequiv m}X_{lm}$$
By Lemma \ref{lem:SW}, the terms of the sum are all equal, say $X_{lm}=c$, if $l\dequiv m$. Then
the sum equals $c$ times the number of $1$'s in the matrix $X^{(\delta)}$, which is exactly the
left-hand side. The result follows.
\end{proof}

\begin{theorem}\label{th:one}
Let $A$ be periodic. Every solution $X\in{\rm Sym}(\V)$ to $AXA^*=X$ is a linear combination
$\displaystyle \sum_{\delta\ge0} c_\delta X^{(\delta)}$ with nonnegative coefficients $c_\delta$.
\end{theorem}
\begin{proof}
Let $X\in{\rm Sym}(\V)$ satisfy $AXA^*=X$. By Lemma \ref{lem:decomp}, 
$$Y=X-\sum_{\delta\ge0}c_\delta X^{(\delta)}$$
where $$c_\delta=\tr XX^{(\delta)}/\tr (X^{(\delta)})^2$$
is an element of ${\rm Sym}(\V)$ satisfying $\sym{A}{Y}=Y$ and in addition
$$ \tr YX^{(\delta)}=0$$
for all $\delta\ge0$. So for any $\delta\ge0$,
$$0=\tr X^{(\delta)}Y=\tr X^{(\delta)}AYA^*=\tr YA^*X^{(\delta)}A$$
And 
\begin{align}
\tr YA^*X^{(\delta)}A&=\sum_{i,j}Y_{ij}A_{\lambda i}X^{(\delta)}_{\lambda \mu}A_{\mu j}\cr
&=\sum_{i\dequiv j} Y_{ij} \sum_{\substack{l\to i\\m\to j}}A_{li}A_{mj}
\end{align}
where the sum of terms involving products of entries from $A$ is positive, since $A$ is irreducible.
 Hence,
$$\sum_{i\dequiv j}Y_{ij}=0 \qquad \Rightarrow\qquad Y_{ij}=0, \ \forall i\dequiv j$$
for all $\delta$. Hence $Y_{ij}=0$ as required.
\end{proof}
\end{subsection}

\subsection{Level 2 correspondence}

Now consider the entries of the upper triangular part of an element of ${\rm Sym}_0(\V)$ as coordinates
of a vector in the space $\spow{\V}{2}\approx\RR^{\binom{n}{2}}$. We identify
$\spow{\V}{2}$ with the linear space over $\RR$ generated by ${\rm Sym}_0(\V)$. The row vector
$$X=(x_{12},x_{13},\ldots,x_{n-1\,n}) \leftrightarrow \hat X\in{\rm Sym}_0(\V)$$
where we use the notation $\mat(X)$ for $\hat X$ as necessary for clarity. So we have
$$\mat(X)_{ij}=(\hat X)_{ij}=\begin{cases} x_{ij}&\text{ if } i<j\cr 0&\text{ if } i=j\cr x_{ji}&\text{ if } i>j\cr \end{cases}
$$

We define the matrix $\spow{A}{2}$ on $\spow{\V}{2}$ corresponding to the action 
\begin{align*}
\spow{A}{2}\colon {\rm Sym}_0(\V)\to{\rm Sym}(\V)\cr
X\mapsto \sym{A}{X}
\end{align*}

The following statements are from \cite{PF}. We will use $\dg X.$ for the column vector.\bigskip

\begin{proposition}\label{propXX}  We have\bigskip

1. $\Mat(\spow{A}{2}\dg  X.) = \sym{A}{\hat X}-D$
where $D$ is a diagonal matrix satisfying $\vstrut\tr D=\tr \sym{A}{\hat X}$.\bigskip

2. If $A$ and $X$ have nonnegative entries, then $D$ has nonnegative entries. In particular, in that case,
vanishing trace for $D$ implies that $D$ vanishes.\bigskip

3. Let $X$ and $A$ be nonnegative. Then 
$$
  \hat X = \sym{A}{\hat X}  \Rightarrow  \spow{A}{2}\dg  X. = \dg  X.
$$
\end{proposition}

\begin{proof} The components of $\spow{A}{2}\dg X.$ are
\begin{align}\label{eq:sym2}
(\spow{A}{2}\dg X.)_{\,ij}  &= \theta_{ij}\theta_{\lambda\mu}
(x_{\lambda\mu}A_{i\lambda }A_{j\mu}+x_{\lambda\mu}A_{i\mu}A_{j\lambda}) \\
 &= \theta_{ij}(\sym{A}{\hat X})_{\,ij}
\end{align}
with the {\sl theta symbol\/} for pairs of single indices
$$\theta_{ij}=\begin{cases} 1, & \text{if } i<j\cr  0, & \text{otherwise} \end{cases}
$$
Note that the diagonal terms of $\hat X$ vanish anyway.
And $\sym{A}{\hat X}$ will be symmetric if $\hat X$ is.
Since the left-hand side has zero diagonal entries, we can remove the theta symbol and compensate by subtracting off
the diagonal, call it $D$. Taking traces yields \#1.
Observe that $D$ has entries
$$
                    D_{\,ii}  = 2\,x_{\lambda\mu}A_{i\lambda }A_{i\mu} 
$$
and \#2 follows directly. For \#3, write
$D= \sym{A}{\hat X}-\Mat(\spow{A}{2}\dg X.)$.
If $\hat X = \sym{A}{\hat X}$, then since $\hat X$ has vanishing trace, $\tr \sym{A}{\hat X}=0$. So
$\tr D=0$, hence $D=0$, by \#2. And $\hat X = \sym{A}{\hat X}=\Mat(\spow{A}{2}\dg X.)$. 
\end{proof}
\begin{remark} Observe from equation \eqref{eq:sym2} that the entries of $\spow{A}{2}$ are the 
permanents of the corresponding $2\times2$ submatrices of $A$.
\end{remark}
\begin{theorem}\label{thmIRPER} \hfill\break\nopagebreak
Let $A$ be irreducible. If $A$ is periodic, then $\det(I-\spow{A}{2}) = 0.$
 
\end{theorem}
\begin{proof}If $A$ is periodic, we have seen that there are solutions to $\sym{A}{X}=X$ with
$X\in{\rm Sym}_0(\V)$, namely the indicators $X^{(\delta)}$ for $\delta>0$. By the third statement
of Proposition \ref{propXX}, the corresponding vectors $\dg X^{(\delta)}.$  
satisfy $\spow{A}{2}\dg  X^{(\delta)}. = \dg  X^{(\delta)}.$. Hence 
$\det(I-\spow{A}{2})= 0$ as required.
\end{proof}
\begin{remark} For the converse and for background details on the approach of this part, see
\cite{PF}.
\end{remark}

Here is an example for illustration.\bigskip

\begin{example}     
Let
$$
  A=\left[ \begin {array}{ccccc} 
	0&1&0&0&0\\
	\noalign{\medskip}0&0&1/2&1/2&0\\
	\noalign{\medskip}0&0&0&0&1\\
	\noalign{\medskip}0&0&0&0&1\\
	\noalign{\medskip}1&0&0&0&0
	\end {array} \right] 
$$
and 
$$\spow{A}{2}=
 \left[ \begin {array}{cccccccccc} 0&0&0&0&1/2&1/2&0&0&0&0
\\\noalign{\medskip}0&0&0&0&0&0&1&0&0&0\\\noalign{\medskip}0&0&0&0&0&0
&1&0&0&0\\\noalign{\medskip}1&0&0&0&0&0&0&0&0&0\\\noalign{\medskip}0&0
&0&0&0&0&0&0&1/2&1/2\\\noalign{\medskip}0&0&0&0&0&0&0&0&1/2&1/2
\\\noalign{\medskip}0&1/2&1/2&0&0&0&0&0&0&0\\\noalign{\medskip}0&0&0&0
&0&0&0&0&0&0\\\noalign{\medskip}0&0&0&1&0&0&0&0&0&0
\\\noalign{\medskip}0&0&0&1&0&0&0&0&0&0\end {array} \right] 
$$
Any left-invariant vector of $A$ is a multiple of
$$
  [\,2, 2, 1, 1, 2\,] 
$$

Solutions $X$ to $\spow{A}{2}\dg X.=\dg X.$ satisfy

\begin{equation}\label{eq:example}
  \hat X= \left[ \begin {array}{ccccc} 
	0&w_{{2}}&w_{{1}}&w_{{1}}&w_{{2}}\\
	\noalign{\medskip}w_{{2}}&0&w_{{2}}&w_{{2}}&w_{{1}}\\
	\noalign{\medskip}w_{{1}}&w_{{2}}&0&0&w_{{2}}\\
	\noalign{\medskip}w_{{1}}&w_{{2}}&0&0&w_{{2}}\\
	\noalign{\medskip}w_{{2}}&w_{{1}}&w_{{2}}&w_{{2}}&0
	\end {array} \right]
\end{equation}
with arbitrary $w_i$. 
The periodic classes are  $\{1\},\{2\},\{3,4\},\{5\}$. 
For a fixed $\delta$, a basic right-invariant vector for $\spow{A}{2}$ is given by $x_{ij}=1$ if ${\rm dist\,}(i,j)=\delta$, $0$
otherwise, with $\delta$ equal to 1 or 2.
\end{example}

\begin{subsection}{Determining the partition classes}
We have seen that the solutions in ${\rm Sym}(\V)$ to $\sym{A}{X}=X$ are linear combinations of
$X^{(\delta)}$ for $\delta\ge0$. The indicator of the partition classes is $X^{(0)}$, the only one with a nonzero
entry on the diagonal. With a symbolic program such as Maple, finding the general solution to 
$(I-\spow{A}{2})\dg X.=0$ on $\spow{\V}{2}$ yields a solution of the form similar to that of equation
\eqref{eq:example}. The blocks along the diagonal are complements of the indicators of the classes 
of the periodic partition. In other words, $i$ and $j$ are in the same partition class if and only if
$X_{ij}=0$ for all solutions $X\in {\rm Sym}_0(\V)$ to $\sym{A}{X}=X$. So an alternative is to solve
$(I-\spow{A}{2})\dg X.=0$ numerically for solutions with the side condition that a particular entry equals 1.
Then adding up a basis of such solutions will exhibit the common zeros corresponding to the partition classes.
\end{subsection}

\begin{subsection}{Distribution of probabilities}
Let $\pi$ denote the invariant measure for the Markov chain. That is, $\pi A=\pi$, with positive entries, and
the eigenvalue 1 of $A$ is simple. So every solution $x$ to $xA=x$ is a multiple of $\pi$.
Let $X\in{\rm Sym}(\V)$ satisfy $\sym{A}{X}=X$. Multiplying on the left by $\pi$ yields
$$\pi AXA^*=\pi X \text{ or }\pi X A^*=\pi X \text{, transposing } 
\implies A\dg(\pi X).=\dg(\pi X).$$
so $\dg (\pi X).$ is a right eigenvector of $A$ with eigenvalue 1, hence a multiple of the all-ones vector $\dg u.$.
Thus, for some positive constant $c$,
$$\pi X=cu$$
Taking $X=X^{(0)}$ shows that each partition class has the same probability $c$. But their disjoint union includes
all vertices and there are $t$ equiprobable classes, thus
$$\pi(P_i)=1/t$$
for each partition class $P_i$. Now consider $X=X^{(\delta)}$, $\delta>0$. A given column, $j$, in $X^{(\delta)}$
has a 1 for every point $\delta$ units from $j$ in the graph. Such points consist either of a single partition class
-- in the case where $t$ is even and $\delta=t/2$ -- or two partition classes. Thus, 
$$\pi X^{(\delta)}=(c_0/t)\,u$$
with $c_0$ equal to 1 or 2 accordingly. For $t$ even we have one class with probability $1/t$ and
$t/2-1$ classes of probability $2/t$, totalling to $1-1/t$. For $t$ odd, we have $(t-1)/2$ classes, all of
probability $2/t$.
\end{subsection}


\begin{acknowledgment}
We thank J.\tsp Kocik for useful consultations.
\end{acknowledgment}

\end{document}